\documentclass[11pt]{article}

\usepackage{amsfonts}
\usepackage{amsmath}

\newcommand{\beq}{\begin{equation}}
\newcommand{\eeq}{\end{equation}}
\newcommand{\bea}{\begin{eqnarray}}
\newcommand{\eea}{\end{eqnarray}}
\newcommand{\beas}{\begin{eqnarray*}}
\newcommand{\eeas}{\end{eqnarray*}}

\newtheorem{theorem}{Theorem}[section]

\newtheorem{proposition}[theorem]{Proposition}
\newtheorem{corollary}[theorem]{Corollary}

\newtheorem{remark}[theorem]{Remark}
\newtheorem{example}[theorem]{Example}
\newtheorem{examples}[theorem]{Examples}
\newtheorem{foo}[theorem]{Remarks}

\newenvironment{proof}{\addvspace{\medskipamount}\par\noindent{\it
Proof}.}
{\unskip\nobreak\hfill$\Box$\par\addvspace{\medskipamount}}

\newcommand{\R}[1]{\mathbb{R}}     

\title{Subelliptic Li-Yau estimates on three dimensional model spaces}

\author{Dominique Bakry \footnote{bakry@math.univ-toulouse.fr},
Fabrice Baudoin \footnote{fbaudoin@math.univ-toulouse.fr},
Michel Bonnefont \footnote{bonnefon@math.univ-toulouse.fr}, Bin Qian\footnote{binqiancn@yahoo.com.cn, This author would like to express
sincere thanks to China Scholarship council for financial support} \\
{\small Institut de Math\'ematiques de Toulouse} \\
{\small Universit\'e de Toulouse} \\
{\small CNRS 5219} \\
}

\begin{document}

\maketitle

\begin{abstract}
We describe three elementary models in three dimensional subelliptic
geometry which correspond to the three models of the Riemannian
geometry (spheres, Euclidean spaces and Hyperbolic spaces) which are
respectively the $SU(2)$, Heisenberg and $SL(2)$ groups. On those
models,  we  prove parabolic Li-Yau inequalities on positive
solutions of the heat equation. We use for that the $\Gamma_{2}$
techniques that we adapt to those elementary model spaces. The
important feature developed here is that although the usual notion
of Ricci curvature is meaningless (or more precisely leads to bounds
of the form $-\infty$ for the Ricci curvature), we describe a
parameter $\rho$ which plays the same r\^ole as the lower bound on
the Ricci curvature, and from which one deduces the same kind of
results as one does in Riemannian geometry, like heat kernel upper
bounds, Sobolev inequalities and diameter estimates.
\end{abstract}

\section{Framework and Introduction}

The estimation of heat kernel measures is a topic which had been
under thorough investigation for the last thirty years at least, see
\cite{LY,davies}. Among the many techniques developed for that, the
famous Li-Yau parabolic inequality \cite{LY} is a very powerful
tool, which relies in Riemannian geometry bounds on the gradient on
heat kernels to lower bounds on the Ricci curvature. More precisely,
in the simplest form, it asserts that, if $E$ is a smooth Riemannian
manifold with dimension $n$ and  non negative Ricci curvature, then
if $f$ is any positive solution of the heat equation
$$\partial_{t} f= \Delta f,$$ where $\Delta$ is the Laplace Beltrami
operator of $E$,
then, if $u= \log f$
$$\partial_{t} u \geq \left|\nabla u\right|^{2}- \frac{n}{2}t.$$

This is a very precise and powerful estimate. For the model case,
which is here the Euclidean space $E=\mathbb{R} ^n$ and when $f$ is
the heat kernel (that is the solution of the heat equation starting
at time $t=0$ from a Dirac mass), then this inequality is in fact an
equality.

From this inequality, one may easily deduce Harnack inequalities and
hence precise bounds on the heat kernel.

Many generalizations of this inequality have been developed, all of
them including lower bounds on the Ricci tensor. In particular, it
works for a general elliptic operator $L$ under the assumption that
it satisfies a curvature-dimension inequality $CD(\rho,n)$, which is
the furthermost generalization on the notion of lower bound on the
Ricci curvature, see \cite{bakry-qian,bakry-ledoux}.

In the non elliptic case, things appear to be infinitely more
complicated. In particular, most of the hypoelliptic systems do not
satisfy any $CD(\rho,n)$ inequality (any reasonable notion of lower
bound on the Ricci tensor leads to the value $-\infty$).
Nevertheless, some Li-Yau inequalities may be obtained \cite{CY}.

In what follows, we shall use the $\Gamma_{2}$ techniques developed
in \cite{bakry-ledoux} to produce these Li-Yau bounds. The method
developed here works quite well on the simple models developed here
(Heisenberg groups, $SU(2)$, $SL(2)$), but could be easily
generalized to a larger class of hypoelliptic operators. We shall
not try to present here the most general results, but concentrate
for simplicity on the three model cases mentioned above. In fact,
they should be thought of as the analogous of the model spaces of
Riemmanian geometry (Euclidean spaces, Spheres and Hyperbolic
spaces).

In all what follows, given an elliptic second order operator $L$ on a smooth
manifold, with no constant term, one defines
$$\Gamma(f,g)= \frac{1}{2}(L(fg)-fLg-gLf)$$ which stands for
$\nabla f\cdot \nabla g$ in the Riemannian case, and the curvature
dimension inequality is defined from the operator $\Gamma_{2}$
$$\Gamma_{2}(f,f) = \frac{1}{2}(\Gamma(f,f)-2\Gamma(f,Lf).$$
Then, $L$ is said to satisfy a $CD(\rho,n)$ inequality if, for any
smooth function $f$, one has
$$\Gamma_{2}(f,f) \geq \rho\Gamma(f,f) + \frac{1}{n}(Lf)^{2}.$$

The parabolic Li-Yau inequality is then described in terms of the
quantity $\left|\nabla f\right|^{2}= \Gamma(f,f)$ and the parameters
$\rho$ and $n$. For the Laplace Beltrami operator $L=\Delta$ on a
smooth Riemannian manifold, this amounts to say that the dimension
is at most  $n$ and that the Ricci curvature is bounded below by
$\rho$. In the hypoelliptic models that we describe below, however,
no such inequality holds (the best possible constant $\rho$ is
$-\infty$), but we shall produce some analogous of the Li-Yau
inequality through a parameter $\rho$ which therefore plays the
r\^ole of a substitute for the Ricci curvature.

In what follows we consider a three-dimensional Lie group $\mathbf{G}$ with Lie algebra $\mathfrak{g}$ and we assume that there is a
 basis $\left\{ X,Y ,Z \right\}$ of $\mathfrak{g}$  such that

$$[X,Y]=Z$$
$$[X,Z]= -\rho Y$$
$$[Y,Z]=\rho X$$

where $\rho \in \mathbb R$.

\begin{example}[$\mathbf{SU}(2)$, $\rho=1$]
The Lie group $\mathbf{SU} (2)$ is the group of
$2 \times 2$, complex, unitary matrices of determinant $1$. Its
Lie algebra $\mathfrak{su} (2)$ consists of $2 \times 2$, complex,
skew-adjoint matrices of trace $0$. A basis of $\mathfrak{su} (2)$
is formed by the Pauli matrices:
 \[
\text{ }X=\frac{1}{2}\left(
\begin{array}{cc}
~0~ & ~1~ \\
-1~& ~0~
\end{array}
\right) ,\text{ }Y=\frac{1}{2}\left(
\begin{array}{cc}
~0~ & ~i~ \\
~i~ & ~0~
\end{array}
\right),
Z=\frac{1}{2}\left(
\begin{array}{cc}
~i~ & ~0~ \\
~0~ & -i~
\end{array}
\right) ,
\]
for which the following relationships hold
\begin{align}\label{Liestructure}
[X,Y]=Z, \quad [X,Z]=-Y, \quad [Y,Z]=X.
\end{align}
\end{example}

\begin{example}[Heisenberg group, $\rho=0$]
The Heisenberg group $\mathbb{H}$ is the group of $3\times3$ matrices:
\[
\left(
\begin{array}
[c]{ccc}
~1~ & ~x~   & ~z ~\\
~0~ & ~1~   & ~y ~\\
~0~ & ~0~   & ~1 ~
\end{array}
\right)  ,\text{ \ }x,y,z\in\mathbb{R}.
\]
The Lie algebra of $\mathbb{H}$ is spanned by the matrices
\[
X=\left(
\begin{array}
[c]{ccc}
~0~ & ~1~ & ~0~\\
~0~ & ~0~ & ~0~\\
~0~ & ~0~ & ~0~
\end{array}
\right)  ,\text{ }Y=\left(
\begin{array}
[c]{ccc}%
~0~ & ~0~ & ~0~\\
~0~ & ~0~ & ~1~\\
~0~ & ~0~ & ~0~
\end{array}
\right)  \text{ and }Z=\left(
\begin{array}
[c]{ccc}%
~0~ & ~0~ & ~1~\\
~0~ & ~0~ & ~0~\\
~0~ & ~0~ & ~0~
\end{array}
\right)  ,
\]
for which the following equalities hold
\[
\lbrack X,Y]=Z,\text{ }[X,Z]=[Y,Z]=0.
\]

\end{example}

\begin{example}[$\mathbf{SL}(2)$, $\rho=-1$]
The Lie group $\mathbf{SL} (2)$ is the group of
$2 \times 2$, real matrices of determinant $1$. Its
Lie algebra $\mathfrak{sl} (2)$ consists of $2 \times 2$ matrices of trace $0$.
A basis of $\mathfrak{sl} (2)$ is formed by the matrices:
 \[
X=\frac{1}{2}\left(
\begin{array}{cc}
~1~ & ~0~ \\
~0~ & -1~
\end{array}
\right)
,\text{ }Y=\frac{1}{2}\left(
\begin{array}{cc}
~0~ & ~1~ \\
~1~ & ~0~
\end{array}
\right),
\text{ }Z=\frac{1}{2}\left(
\begin{array}{cc}
~0~ & ~1~ \\
-1~& ~0~
\end{array}
\right)
,
\]
for which the following relationships hold
\begin{align}\label{Liestructure}
[X,Y]=Z, \quad [X,Z]=Y, \quad [Y,Z]=-X.
\end{align}
\end{example}

We consider on the Lie group $\mathbf{G}$ the subelliptic, left-invariant, second order differential operator
$$L= X^{2}  + Y^{2},$$
as well as the heat semigroup
$$ P_t=e^{t L}.$$
We also set $$\Gamma(f,f) =\frac{1}{2}(Lf^2-2fLf)= (Xf)^{2}+ (Yf)^{2},$$ and
$$\Gamma_{2} = \frac{1}{2}(L\Gamma(f,f) - 2\Gamma(f, Lf)).$$

In the present setting,
\begin{align}\label{Gamma2}
\Gamma_2 (f,f)= (X^2f)^2+(Y^2f)^2+\frac{1}{2} \left( (XY+YX)f \right)^2+ \frac{1}{2}(Zf)^2 +
\rho \Gamma (f,f)-2 (Xf)(YZf)+2(Yf)(XZf).
\end{align}

The mixt terms $-2 (Xf)(YZf)+2(Yf)(XZf)$ prevents to find any lower
bound on this quantity involving $\Gamma(f,f)$ and $(Lf)^{2}$ only,
whence the absence of any $CD(\rho,n)$ inequality.

\section{Li-Yau type estimates for the heat semigroup}

The classical method of Li and Yau \cite{LY} consists in applying
the maximum principle to a carefully chosen expression. The method
developed in \cite{bakry-ledoux} is quite different. Considering a
positive solution of the heat equation $\partial_{t} f = Lf$, and
denoting $f \mapsto P_{t}f$ the associated heat kernel, one writes
$u = \log f$ and look at the expression
$$ \Phi(s) = P_{s}(f(t-s) \Gamma(u(t-s), u(t-s))),$$ defined  for $0<s<t$. Then, one obtains
through the $CD(\rho,n)$ inequality a differential inequality
$$\Phi'(s) \geq (A\Phi(s)+B)^{2} + C,$$ where $A, B, C$ are
expressions which are constant in $t$ but may depend on the function
$f$. Then, the parabolic Li-Yau inequality is obtained as a
consequence of this differential inequality.

Here, we shall develop this method a bit further, looking at more
complicated quantities like
$$P_{s}(f(t-s)(a(s)\Gamma(u(t-s), u(t-s))+ b(s) (Zu(t-s))^{2})),$$
and try to get some differential inequality on it. The computations
developed here are not restricted to Lie group, since we only use an
generalized $CD(\rho,n)$ inequality. There are many hypoelliptic
systems that may be treated under the same lines. The reason why we
restrict ourselves to those model cases described previously  are
mainly for pedagogical reasons.

We have the following inequality, which is our technical starting
point:
\begin{proposition}\label{inegdiff}
Let $f:\mathbf{G} \rightarrow \mathbb{R}$ be positive. Let $t >0$,
for all $x \in \mathbf{G}$ and $s\in [0,t]$, consider the
expressions
\[
\Phi_1 (s)=P_s ((P_{t-s} f) \Gamma (\ln P_{t-s} f))(x)
\]
and
\[
\Phi_2 (s)=P_s ((P_{t-s} f) (Z \ln P_{t-s} f)^2)(x).
\]
Then, for every differentiable, non-negative and decreasing function $b:[0,t] \rightarrow \mathbb{R}$,
\[
\left( - b' \Phi_1 +b \Phi_2 \right)' (s)
\ge - b'(s) \left( \left( \frac{b''(s)}{b'(s)}+2 \frac{b'(s)}{b(s)} + 2 \rho \right) LP_t f (x)-\frac{1}{4}\left( \frac{b''(s)}{b'(s)} +2 \frac{b'(s)}{b(s)} + 2 \rho\right)^2 P_t f (x)\right).
\]
\end{proposition}

\begin{proof}
We fix a positive function $f$, $t>0$ and we perform all the
following computations  at a given point $x$.

With the same notations as Proposition \ref{inegdiff},
straightforward (but quite tedious) computations show that
\[
\Phi'_1 (s)=2P_s ((P_{t-s} f) \Gamma_2 (\ln P_{t-s} f))
\]
and
\[
\Phi'_2 (s)=2P_s ((P_{t-s} f) \Gamma(Z \ln P_{t-s} f)).
\]

For the last equality we use the crucial facts that
$$[L,Z]=0$$ and $$X(f)Z(f)[X,Z](f)+ Y(f) Z(f)[Y,Z](f)=0.$$

Now, thanks to the Cauchy-Schwarz inequality, the expression (\ref{Gamma2}), shows that for every $\lambda >0$,  and every smooth function $g$,
\[
\Gamma_2 (g) \ge \frac{1}{2} (Lg)^2 +\frac{1}{2} (Zg)^2 +\left(\rho-\frac{1}{\lambda} \right) \Gamma (g)-\lambda \Gamma (Zg).
\]
We therefore obtain the following differential inequality
\[
\Phi'_1 (s) \ge P_s ((P_{t-s}f) (\mathcal{L} \ln P_{t-s}f)^2) + \Phi_2 (s) +\left(2\rho-\frac{2}{\lambda} \right)\Phi_1 (s) -\lambda \Phi_2'(s).
\]
We now have that for every $\gamma \in \mathbb{R}$,
\[
(L \ln P_{t-s} f)^2 \ge 2 \gamma L \ln P_{t-s}f - \gamma^2,
\]
and
\[
\mathcal{L} \ln P_{t-s}f =\frac{L P_{t-s} f}{P_{t-s}f} -\frac{\Gamma (P_{t-s}f)}{(P_{t-s}f)^2 }.
\]
Thus, for every $\lambda >0$ and every $\gamma \in \mathbb{R}$,
\[
\Phi'_1 (s) \ge \left(2 \rho-\frac{2}{\lambda} -2 \gamma \right) \Phi_1 (s) + \Phi_2 (s) -2 \lambda \Phi_2' (s) +2 \gamma \mathcal{L}P_t f - \gamma^2 P_t f.
\]
Now for two functions $a$ and $b$ defined on the time interval $[0,t)$ with $a$ positive,  we have
$$(a \Phi_1 +b \Phi_2)' \geq \left( a' +(2 \rho -\frac{2}{\lambda}- 2 \gamma ) a\right) \Phi_1
+ (a+b') \Phi_2 + (-a \lambda +b) \Phi_2 ' + 2 a \gamma L P_t f-a\gamma^2 P_t  f.$$

So, if $b$ is a positive decreasing function on the time interval $[0,t)$, by choosing in the previous inequality
$$a=-b',
$$
\[
\lambda=- \frac{b}{b'},
\]
and
\[
\gamma=\frac{1}{2} \left( \frac{b''}{b'}+2 \frac{b'}{b} + 2 \rho \right),
\]
we get the desired result.
\end{proof}
As a first corollary, by using the function
\[
b(s)=\left( t-s\right)^\alpha, \quad \alpha >2
\]
and integrating from 0 to $t$, we deduce
\begin{corollary}\label{Li-Yau}
For all $\alpha>2$, for every positive function $f$ and $t >0$,
\[
\Gamma (\ln P_t f) +\frac{t}{\alpha}  (Z \ln P_t f)^2  \leq \left(
\frac{3 \alpha -1}{\alpha -1} -\frac{2 \rho t}{\alpha} \right)
\frac{LP_t f}{P_t f} + \frac{ \rho^2 t}{\alpha}  -
\frac{\rho(3\alpha-1)}{\alpha-1} + \frac{(3\alpha-1)^2}{\alpha-2}
\frac{1}{t}.\]
\end{corollary}

Observe that this takes a simpler form when $\rho\geq0$, since then
one can use  proposition \ref{inegdiff} with $\rho=0$ and get

\begin{corollary}\label{Li-Yau3}

    When $\rho\geq0$, there exist constants $A,B$ and $C$ such that,
    with $u= \ln (P_{t} f)$
    $$\partial_{t} u \geq A\Gamma(u) + Bt(Zu)^{2}  -\frac{C}{t}.$$

   \end{corollary}

   In particular, one gets $\partial_{t}u \geq -C/t$, which gives
   $$P_{t}f \leq t^{-C}P_{1}f.$$

   On the Heisenberg group, one sees that the behavior of $P_{t}f$
   when $t$ goes to $0$ is of order $t^{-2}$ (a simple dilation
   argument shows that). Therefore, one sees that the optimal
   constant $C$ in the previous inequality is $C=2$. Unfortunately,
   it can be shown by some elementary considerations similar to those
    developed in the proof of corollary \ref{tps_long} that the best
   constant one may obtain from the previous proposition shall always
   produce a constant $C>2$. This is a strong difference with the
   classical parabolic Li-Yau inequality  where the inequality
   $$\partial_t u \geq -\frac{n}{2t}$$ gives the right order of
   magnitude of the heat kernel near $t=0$.

Now when $ \rho >0$, we easily get an exponential decay by using the function:
\[
b(s)=\left( e^{-\frac{2\rho s}{3 \alpha }} - e^{-\frac{2\rho t}{3 \alpha }} \right)^\alpha, \quad \alpha >2.
\]
This writes:
\begin{corollary}\label{Li-Yau2}
For every $\alpha>2$, for every positive function $f$, $x \in \mathbf{G}$ and $t >0$,
\[
\Gamma (\ln P_t f)(x) +\frac{3}{2} \frac{1-e^{-\frac{2\rho t}{3\alpha}}}{\rho}  (Z \ln P_t f)^2 (x)  \leq
 \frac{3 \alpha -1}{\alpha -1}e^{-\frac{2\rho t}{3\alpha}}  \frac{LP_t f (x)}{P_t f (x)}
+  \frac{3}{2} \rho \frac{\left(
1-\frac{1}{3\alpha}\right)^2}{1-\frac{2}{\alpha}}
\frac{e^{-\frac{4\rho t}{3\alpha}}}{1-e^{-\frac{2\rho
t}{3\alpha}}}.\]
\end{corollary}





Moreover for $\rho>0$ and $t$ large, with more work we actually can do better.

\begin{corollary}\label{tps_long}
Let us assume $\rho >0$. There exist $t_0 >0$ and $C>0$, such that for any positive function $f$,
$$ |\partial_t \ln P_t f (x)| \leq  C \exp\left(-\frac{\rho t }{3}\right), \quad t \ge t_0, x \in \mathbf{G}.
$$
\end{corollary}

\begin{proof}
To make this proof we have to be more precise in the study of the differential inequality of Theorem \ref{inegdiff}. Start with this inequality and set $V(b)=-b^2 b'$ for $b$ a positive decreasing function such that $b(t)=b'(t)=0$.
 The constraints   that the non negative function
    $V$ on $[0, b_{0}]$ must satisfy are
     $$t= \int_{0}^{b_{0}} \frac{x^{2}}{V(x)}dx$$ and
    $$\left(\frac{V(x)}{x^2}\right)_{x=0}=0.$$

    We then  get with $u_t= \ln P_t f$  and $a_{0}= \frac{V(b_{0})}{b_{0}^{2}}$
    $$a_{0}\Gamma(u_{t}) + b_{0}(Zu_{t})^{2} \leq A \partial_{t}
    u_{t} + B,$$ where for any choice of such a function
    $V$,
     one has$$
     A= \int_{0}^{b_{0}}
    \left(\frac{V'}{x^{2}}- 2 \rho t\right) dx, ~B= \frac{1}{4}\int_{0}^{b_{0}}
    \left(\frac{V'}{x^{2}}-2 \rho t\right)^2dx.
   $$

   In this system, we see that changing $V(s)$ into $\frac{V(\lambda
      s)}{\lambda^{3}}$ and $b_{0}$ into $\frac{b_{0}}{\lambda}$
      leaves $t$ unchanged and multiply every constant $a_{0}$, $A$
      and $B$ by $\frac{1}{\lambda}$. Therefore, we may assume
      that $b_{0}=1$ without any loss. Also,  changing $V(s)$ into $cV(s)$ allows us to reduce to the
      case $t=1$.  So finally we have rephrased the problem as follows.
       For any non negative function $V$ on $[0,1]$ such that
    $$\int_{0}^{1} \frac{x^{2}}{V}dx=1,
    ~\left(\frac{V(x)}{x^{2}}\right)_{x=0}=0, $$ and for any $u= \log P_{t}
    f$ with $f\geq0$
    one has
    $$V(1) \Gamma(u) + t (Zu)^{2} \leq (\alpha(V)-2 \rho t)\partial_{t}u +
    \frac{1}{4t}\left(\beta(V)-\alpha^{2}(V) + (\alpha(V)-2 \rho t)^{2}\right) ,$$ where
    $$\alpha(V) = \int_{0}^{1} \frac{V'}{x^{2}}dx, ~\beta(V)=
    \int_{0}^{1}\left(\frac{V'}{x^{2}}\right)^{2} dx.$$

    The preceding calculus is valid for any $\rho$. An easy integration by parts shows us the term $\alpha(V)$ is non negative whatever $V$ is. But now for $\rho>0$,  observe that this time the term $\alpha(V)- 2 \rho t$ can be made negative,  and therefore we may get as in the elliptic case with
     strictly positive Ricci bound a universal upper bound on
     $|\partial_{t} u|$.

One has the obvious inequalities
    $$\alpha(V) > V(1)- \left(\frac{V(x)}{x^{2}}\right)_{x=0}+ 8, ~ \beta(V) > \alpha(V)^{2},$$ and in the
    previous, no equality may occur (in the first one because then
    $\beta= \infty$ and in the second one because of the constraint
    on $V$.) The first inequality comes from
    $$\int_{0}^{1} \frac{V'}{x^{2}}dx = V(1)+ 2\int_{0}^{1}
    \frac{V}{x^{3}}dx,$$ and
    $$\int_{0}^{1} \frac{V}{x^{3}}dx \int_{0}^{1} \frac{x^{2}}{V }dx
    \geq \left(\int _{0}^{1} \frac{dx}{\sqrt{x}}\right)^{2} = 4.$$

    To make the term $\beta(V)-\alpha^{2}(V)$ small we are lead to
    choose $V= \lambda x^{3}$ on $[\epsilon,1]$
   and $V= \lambda \epsilon^{3-\gamma}x^{\gamma}$ on $[0, \epsilon]$,
   for some fixed $\gamma\in (5/2,3)$.
    The constraint on $V$ implies
    $$\lambda  = -\log \epsilon + \frac{1}{3-\gamma}.$$
    Meanwhile, we have

    $$\alpha = \lambda\left(3+2\epsilon\frac{3-\gamma}{\gamma-2}\right),$$ and
    $$\beta = \lambda^{2}\left(9+
    \epsilon\frac{(15-\gamma)(3-\gamma)}{2\gamma-5}\right),$$ so that
    $$\beta-\alpha^{2} = \lambda^{2} \epsilon
    \frac{(3-\gamma)^{2}}{\gamma-2}\left(\frac{\gamma+10}{2\gamma-5} +
    \epsilon\frac{4}{\gamma-2}\right).$$
    By taking $$\epsilon = \exp\left(-\frac{2 \rho}{3}t + \frac{1}{3-\gamma } + R \right)$$ for $t$ large enough to ensure $\varepsilon <1$, one obtains
    $$ \alpha - 2 \rho t \simeq -3 R$$
    and
    $$\beta-\alpha^{2} \simeq C t^2 \varepsilon \simeq C t^2
    \exp\left(-\frac{2\rho t }{3}\right).
    $$
    With $R=c t \exp(-\frac{\rho }{3})$ the  terms $(\alpha - 2 \rho t)^2$ and $\beta-\alpha^{2}$ are of the same order and
     playing now with the sign of $c$, one gets
     $$\left|\partial_{t} u\right| \leq C\exp\left(-\frac{t}{3}\right).$$
\end{proof}

Interestingly, only from these estimates, we can deduce that for
$\rho>0$ the Lie group $\mathbf{G}$ has to be compact. (This is of
course not new since the Lie algebra is that of a compact
semi-simple Lie group). But we also get an upper bound on the
diameter similar to the classical upper bound of the Myers's
theorem, together with some precise information on the Sobolev
constants and the spectral gap. Those considerations in fact show
that this parameter $\rho$ may serve as a substitute of the Ricci
lower bound for a Riemannian manifold.  We proceed first by showing
that in that case there is a spectral gap.

\begin{proposition}
Let us assume $\rho >0$. The spectrum of $-L$ lies in $\{ 0 \} \cup [ \frac{\rho}{3} , +\infty ]$.
\end{proposition}

\begin{proof}
We fix $x \in \mathbf{G}$ and denote by $p_t (x,\cdot)$ the heat kernel starting from $x$. We have for $t \ge t_0$,
\begin{align}\label{estimeenoyau}
\mid \partial_t \ln p_t (x,y) \mid \le  C \exp\left(-\frac {\rho
t}{3}\right).
\end{align}
This shows us that  $\ln p_t $ converges when $t\to\infty$. Let us
call  $\ln p_\infty $ this limit. Moreover, from Corollary
\ref{Li-Yau2},  $\Gamma(\ln p_t)$ is bounded above by a constant
    $C(t)$ which goes to $0$ when $t$ goes to $\infty$. Since
    the oscillation between $\ln p_t(x,y_1)$ and $\ln p_t(x,y_2)$ is bounded above
    by $\sqrt{C(t)} d(y_1,y_2)$, for the associated  Carnot-Carath\'eodory
    distance, which may be defined (see \cite{bakry-st-flour}) as
    \beq \label{distance} d(x,y)= \sup_{\{f, \Gamma(f,f) \leq1\}}
    f(x)-f(y), \eeq such that  if
    $\Gamma(f, f)\leq C$,  then $f(x)-f(y) \leq \sqrt{C}d(x,y)$.

    In
    the limit, $\ln p_{\infty}( x, \cdot)$ is a constant. We deduce from this that the
    invariant measure $\mu$ is finite. We may then
    as well suppose that this measure is a probability, in which case
    $p_{\infty}=1$. By integrating the inequality (\ref{estimeenoyau}) from $t$ to $\infty$ we therefore obtain for $t \ge t_0$:
\[
\mid \ln p_t (x,y) \mid \le C_2 \exp\left(-\frac {\rho t}{3}\right)
\]
and thus
\[
\exp\left( -C_2 \exp\left(-\frac {\rho t}{3}\right)\right) \le p_t
(x,y) \le \exp \left( C_2 \exp\left(-\frac {\rho
t}{3}\right)\right).
\]
 This implies by the Cauchy-Schwarz inequality that for $f \in L^2 (\mu)$ such that $\int f d\mu =0$,
 \[
 (P_t f)^2 \le C_3 \exp\left(-\frac {2\rho t}{3}\right) \int f^2 d\mu.
 \]
For a symmetric Markov semigroup $P_{t}$, this is a standard fact
(see \cite{bakry-st-flour} for example) that this is equivalent to
say that the spectrum of $-L$ lies in $\{0\}\cup [\rho/3, \infty)$,
or equivalently that we have a spectral gap inequality: for any
function $f$ in $L^{2}$ such that $\nabla f$ is in $L^{2}$, one has
\beq \label{SGInequality}\int f^{2}d\mu \leq \left(\int
fd\mu\right)^{2} + \frac{3}{\rho}\int \left|\nabla f\right|^{2}
d\mu. \eeq
\end{proof}
\begin{remark}
It can be shown that the spectral gap is actually $\frac{\rho}{2}$
and not $\frac{\rho}{3}$.
\end{remark}
We can now conclude with a substitute of the Myers's theorem:
\begin{proposition}
Assume that $ \rho >0$, then the diameter of $\mathcal{L}$ for the
Carnot-Caratheodory distance is finite.
\end{proposition}

\begin{proof}
    We are now going to prove a Sobolev inequality for the invariant measure $\mu$.
    Indeed, for $0<t\leq t_{0}$ we have
    $$\partial_{t} \ln p_t \geq -C/t,$$ from which we get
    $$\ln p_{t_{0}}-\ln p_{t} \geq -C\log (t_{0}/t),$$ and therefore
    $$\ln p_t \leq A - C \log t$$ where $A$ is a constant.
    This gives the ultracontractivity of the semigroup $P_t$ with a
    polynomial bound $t^{-C}$ when $t\to 0$.

   Now it is a well known fact (see \cite{Varopoulos, bakry-st-flour}) that this last property
   is equivalent to a Sobolev inequality
   \beq\label{SobIn}\left(\int f^{\frac{2C}{C-1}}d\mu\right)^{\frac{C-1}{C}}\leq A\int f^{2}
   d\mu + B \int\left\| \nabla f\right\|^{2} d\mu. \eeq

   When we have both Sobolev inequality (\ref{SobIn}) and spectral
   gap inequality (\ref{SGInequality}) then  (see
   \cite{bakry-st-flour}) we have a tight Sobolev
   inequality, that is the Sobolev inequality (\ref{SobIn}) with
   $A=1$.

   In this situation, the diameter of $E$ with respect to the
   distance defined in \ref{distance} is finite (see
   \cite{bakry-ledoux-myers}), which concludes the proof.

\end{proof}

\end{document}